\documentclass[leqno,11pt]{siamltex}
\usepackage{amsfonts}
\usepackage{amsmath}
\usepackage{amssymb}
\usepackage{subfig}
\usepackage{graphicx,version,color}
\usepackage{bm}
\usepackage{xcolor}
\usepackage{multirow}

\catcode`\@=11
\@addtoreset{equation}{section} 
\newcommand{\td}{{\rm d}}
\newcommand{\tb}{{\rm b}}
\newcommand{\ttr}{{\rm tr}}

\newtheorem{assum}{Assumption}[section]

\title{Deep Ritz method for the spectral fractional Laplacian equation using the Caffarelli-Silvestre extension}
\author{Yiqi Gu\thanks{Department of Mathematics, The University of Hong Kong, Pokfulam, Hong Kong ({\tt yiqigu@hku.hk, mng@maths.hku.hk}). This work is supported by HKRGC GRF 12300218, 12300519, 17201020 and 17300021.} \and Micheal K. Ng\footnotemark[1]}

\begin{document}

\maketitle

\begin{abstract}
In this paper, we propose a novel method for solving high-dimensional spectral fractional Laplacian equations. Using the Caffarelli-Silvestre extension, the $d$-dimensional spectral fractional equation is reformulated as a regular partial differential equation of dimension $d+1$. We transform the extended equation as a minimal Ritz energy functional problem and search for its minimizer in a special class of deep neural networks. Moreover, based on the approximation property of networks, we establish estimates on the error made by the deep Ritz method. Numerical results are reported to demonstrate the effectiveness of the proposed method for solving fractional Laplacian equations up to ten dimensions. Technically, in this method, we design a special network-based structure to adapt to the singularity and exponential decaying of the true solution. Also, A hybrid integration technique combining Monte Carlo method and sinc quadrature is developed to compute the loss function with higher accuracy.
\end{abstract}

\begin{keywords}
Ritz method; Deep learning; Fractional Laplacian; Caffarelli-Silvestre extension; Singularity;
\end{keywords}
\begin{AMS}
65N15; 65N30; 68T07; 41A25
\end{AMS}

\pagestyle{myheadings}
\thispagestyle{plain}
\markboth{}{}

\section{Introduction}\label{Sec_Introduction}
As a nonlocal generalization of the Laplacian $-\Delta$, the
spectral fractional Laplacian $(-\Delta)^s$ with a fraction
$s$ arises in many areas of applications, such as anomalous diffusion \cite{Carreras2001,Uchaikin2015}, turbulent flows \cite{Shlesinger1987}, L\'{e}vy processes \cite{Jourdain2005}, quantum mechanics \cite{Laskin2000}, finance \cite{Tankov2003,Cont2005} and pollutant transport \cite{Vazquez2014}. In this paper, we develop a network-based Ritz method for solving fractional Laplacian equations using the Caffarelli-Silvestre extension.

Let $d\in\mathbb{N}^+$ be the dimension of the problem and $\Omega$ be a bounded Lipschitz domain in $\mathbb{R}^d$. Also, suppose $0<s<1$ and $f$ is a function defined in $\Omega$, we consider the following spectral fractional Laplacian equation with homogeneous Dirichlet condition,
\begin{equation}\label{01}
\begin{cases}
(-\Delta)^s U(x) = f(x),\quad\forall x\in\Omega,\\
U(x) = 0,\quad \forall x\in\partial\Omega,
\end{cases}
\end{equation}
where $(-\Delta)^s$ is defined by the spectral decomposition of $-\Delta$ with the same boundary conditions. More precisely, we suppose the countable set $\{(\lambda_n,\phi_n)\}_{n=1}^\infty$ are all the eigenvalues and orthonormal eigenfunctions of the following problem,
\begin{equation}\label{26}
\begin{cases}
-\Delta \phi = \lambda \phi,\quad\text{in}~\Omega,\\
\phi = 0,\quad\text{on}~\partial\Omega,\\
(\phi,\phi)=1,
\end{cases}
\end{equation}
where $(\cdot,\cdot)$ is the standard inner product in $L^2(\Omega)$. Then for any $U\in L^2(\Omega)$,
\begin{equation}
(-\Delta)^s U=\sum_{n=1}^{\infty}\tilde{U}_n\lambda_n^s\phi_n,~\text{with}~\tilde{U}_n:=(U,\phi_n).
\end{equation}

We remark that another definition of the fractional Laplacian is formulated by integrals with non-local structures, and these two definitions do not coincide. It is difficult to solve fractional Laplacian equations of either definition directly using numerical methods for regular differential equations (e.g., finite difference method or finite element method) due to the non-local property of the fractional operator \cite{Bonito2019_2,Acosta2017}. Instead, one effective approach is to use the Caffarelli-Silvestre extension \cite{Caffarelli2007}. Specifically, let us introduce a scalar variable $y$ and consider the following $d+1$-dimensional problem
\begin{equation}\label{02}
\begin{cases}
\nabla\cdot\left(y^\alpha \nabla u(x,y)\right) = 0,\quad\forall (x,y)\in\mathcal{D}:=\Omega\times(0,\infty),\\
-\underset{y\rightarrow0}{\lim}~y^\alpha u_y(x,y)=d_sf(x),\quad\forall x\in\Omega,\\
u(x,y) = 0,\quad \forall (x,y)\in\partial_L\mathcal{D}:=\partial\Omega\times[0,\infty),
\end{cases}
\end{equation}
where $\alpha=1-2s$ and $d_s=2^{1-2s}\Gamma(1-s)/\Gamma(s)$. Suppose $u(x,y)$ solves \eqref{02}, then $U(x):=u(x,0)$ is a solution of \eqref{01} \cite{Nochetto2015}. Consequently, one can solve the extended problem \eqref{02} with regular derivatives to avoid addressing spectral fractional differential operators, with the extra cost that (i) the dimension is increased from $d$ to $d+1$; (ii) the domain is extended from the bounded one $\Omega$ to an unbounded cylinder $\partial_L\mathcal{D}$. Several methods have been proposed for \eqref{02}, such as the tensor product finite elements \cite{Nochetto2015} and the enriched spectral method using Laguerre functions \cite{Chen2020}. We remark that the Caffarelli-Silvestre extension is exclusively for the spectral fractional Laplacian and not for the integral fractional Laplacian. Moreover, the extension technique can be extended to more general fractional symmetric elliptic operators of the form $-\nabla\cdot(A(x)\nabla U(x))+C(x)U(x)$ with $A(x)$ being symmetric and positive definite and $C(x)$ being positive.

However, conventional linear structures such as finite elements and polynomials are usually incapable of high-dimensional approximation in practice. For example, suppose a tensor product linear structure has $\tilde{N}$ basis functions in each dimension, then the total degree of freedom is $O(\tilde{N}^d)$, which is a huge number if $d$ is moderately large. Such a curse of dimensionality prevents one from using linear algebra structures in high-dimensional problems with $d>3$. For the spectral fractional Laplacian, most existing methods based on the Caffarelli-Silvestre extension could solve numerical examples of dimension at most two, mainly due to the limitation of storage. Our primary target is to solve many physically relevant problems that appear in three or higher dimensional situations.

In recent years, deep neural networks (DNNs) are widely studied and utilized in applied problems. As a composition of simple neural functions, the DNN has parameters nonlinearly arrayed in the network structure. For a fully-connected DNN with depth $L$, width $M$ and dimension of inputs $d$, the total number of parameters is of $O(LM^2+Md)$. Therefore, the degree of freedom increases linearly with the dimension and DNNs are capable of dealing with high-dimensional approximation in practice. Theoretically, it is shown that DNNs have decent approximation properties for particular function spaces (e.g., Barron space). The seminal work of Barron \cite{Barron1992,Barron1993} deduces $L^2$-norm
and $L^{\infty}$-norm approximations of two-layer sigmoid networks. Recent work \cite{Klusowski2018,E2019,E2020,E2020_2,Siegel2020_1,Siegel2020_2,Caragea2020,Lu2021} considers more variants of the network-based approximation for Barron-type spaces. Generally, given a Barron function $g:\Omega\rightarrow\mathbb{R}$, there exists a two-layer neural network $g_M$ with width $M$ and common-used activations such that
\begin{equation}\label{11}
\|g-g_M\|_\Omega\leq O\left(\frac{\|g\|_\mathcal{B}}{\sqrt{M}}\right),
\end{equation}
where $\|g\|_\mathcal{B}$ is the Barron norm of $g$, and $\|\cdot\|_\Omega$ can be $L^2(\Omega)$, $H^1(\Omega)$ or $L^\infty(\Omega)$ norm under different hypothesis.
It is worth mentioning that the above error bound is independent of the dimension of the input variable; hence the network-based approximation can overcome the curse of dimensionality.

In this paper, we solve \eqref{02} by Ritz method in which DNNs are employed to approximate the solution. More precisely, we reformulate \eqref{02} as a minimal Ritz energy functional problem and characterize the Sobolev space of the weak solution. Next, as a subset of the solution space, a class of DNNs is taken as the hypothesis space of the minimization. We design a special network structure for the DNN class such that (i) it satisfies the homogeneous boundary condition on
$\partial_L\mathcal{D}$ in \eqref{02}; (ii) it decays to zero exponentially as $y\rightarrow\infty$; and (iii)
it has a singularity behaves as $O(y^{k+1-\alpha})$ for integers $k$ at $y=0$. Note that the second and third properties mentioned above are also satisfied by the true solution. Consequently, the special DNNs have better approximation performance than generic DNNs, which is also observed in our numerical experiments.

Theoretically, under a Barron-type framework, we investigate the approximation error between the special DNN class and the solution space under the Sobolev norm, which has a similar form to \eqref{11}. Based on that, we estimate the solution error of the proposed Ritz method using the special DNN class, assuming that the true solution has components in the Barron space. The final solution error is of $O(M^{-1/2})$, which is consistent with the approximation error. We remark that the error bound of our method is advantageous over the existing methods \cite{Nochetto2015,Chen2020} using finite element or Laguerre functions if the dimension is moderately large, since the error order $-1/2$ is independent of the dimension. Also, the error order is consistent with the deep Ritz method for regular Laplacian equations \cite{Muller2021}

In the implementation, a combination of stochastic and deterministic methods is employed to compute the integrals in the energy functional. Specifically, we utilize the quasi-Monte Carlo method and the sinc quadrature rule \cite{Lund1992} to evaluate the integrals in terms of $x$ and $y$, respectively. For the former, due to the potentially high dimension of $x$, Monte Carlo-type methods are effective and easy to implement. For the latter, although the integrand in terms of $y$ is one-dimensional, it has a singular term $y^\alpha$ when $\alpha\neq0$. While sinc quadrature is highly accurate for integrals with fractional powers and therefore preferred here. By numerical experiments, we demonstrate that our method can solve model problems up to $d=10$ with desired accuracy. To the best of our knowledge, this is the first attempt to solve high-dimensional fractional Laplacian equations by deep learning methods.

Overall, the highlights of our work can be summarized as follows:
\begin{itemize}
  \item Development of a special approximation structure based on generic DNNs according to the special properties of the true solution;
  \item Combination of stochastic Monte Carlo method for high dimensions and deterministic sinc quadrature for high accuracy in the learning process;
  \item Simulation of 10-D fractional Laplacian equations with relative error $O(10^{-2})$.
\end{itemize}

The rest of the paper is organized as follows. In Section 2, we reformulate the problem \eqref{02} as the minimization of an energy functional and show their equivalence. In Section 3, the fully connected neural networks are introduced. We characterize the special structures of the hypothesis space and discuss its approximation property. In Section 4, we derive the error estimate for the proposed method. Numerical experiments are presented to show the effectiveness of our method in Section 5. Finally, some
concluding remarks are given in Section 6.

\section{Minimization of Energy Functional}
We solve the regular partial differential equation \eqref{02} under the framework of Ritz method. The equation can be transformed to an equivalent minimal functional, and we look for Sobolev weak solutions instead of classical solutions. Similarly, one can also solve \eqref{02} using Galerkin method by introducing appropriate test spaces such as in \cite{Nochetto2015,Chen2020}. Since learning-based methods aim to find solutions via optimization, the use of Ritz method can be achieved
for building such formulation.

\subsection{The space of weak solutions}
Let $\mathcal{Z}$ be any region and $\omega$ be a positive weight function. We define the weighted $L^2$ space as
\begin{equation}
L^2_\omega(\mathcal{Z}):=\left\{v~\Big|~\int_\mathcal{Z}|v(z)|^2\omega(z)\td z<\infty\right\},
\end{equation}
equipped with the inner product
\begin{equation}
\left(v_1,v_2\right)_{\omega,\mathcal{Z}}:=\int_\mathcal{Z}v_1(z)v_2(z)\omega(z)\td z,\quad\forall v_1,v_2\in L^2_\omega(\mathcal{Z}),
\end{equation}
and the induced norm
\begin{equation}
\|v\|_{\omega,\mathcal{Z}}:=\left(v,v\right)_{\omega,\mathcal{Z}}^\frac{1}{2}.
\end{equation}
The weight $\omega$ is omitted from the notations if $\omega\equiv 1$.

We also define the weighted Sobolev space as
\begin{equation}
H^1_\omega(\mathcal{Z}):=\left\{v~|~v\in L^2_\omega(\mathcal{Z}),\nabla v\in L^2_\omega(\mathcal{Z})\right\},
\end{equation}
equipped with the norm
\begin{equation}
\|v\|_{1,\omega,\mathcal{Z}}:=\left(\|v\|_{\omega,\mathcal{Z}}^2+\|\nabla v\|_{\omega,\mathcal{Z}}^2\right)^\frac{1}{2}.
\end{equation}

It is shown in \cite{Miller2001} the solution of the extended problem \eqref{02} has a desirable property that it converges exponentially to zero as $y\rightarrow\infty$. Therefore we can define the solution space as
\begin{equation}
H^{1,\tb}_{y^\alpha}(\mathcal{D}):=\left\{v\in H^1_{y^\alpha}(\mathcal{D})~\Big|~\underset{y\rightarrow\infty}{\lim}v(x,y)=0,~v(x,y)|_{\partial_L\mathcal{D}}=0\right\},
\end{equation}
with the norm
\begin{equation}
\|v\|_{H^{1,\tb}_{y^\alpha}(\mathcal{D})}=\|\nabla v\|_{y^\alpha,\mathcal{D}}.
\end{equation}

Denote the trace for all $v\in H^{1,\tb}_{y^\alpha}(\mathcal{D})$ by
\begin{equation}\label{25}
\ttr\{v\}(x)=v(x,0).
\end{equation}

Moreover, for column vectors or vector-valued functions, we use $|\cdot|$ to denote their Euclidean norm.

\subsection{Minimal energy functional}
We aim to characterize the solution of \eqref{02} as a minimizer of an corresponding energy functional. For this, we define the functional
\begin{equation}
\mathcal{I}[w]:=\frac{1}{2}(\nabla w,\nabla w)_{y^\alpha,\mathcal{D}}-d_s\left(f,\ttr\{w\}\right)_\Omega,\quad\forall w\in H^{1,\tb}_{y^\alpha}(\mathcal{D}).
\end{equation}

We have the following result.
\begin{theorem}\label{Thm01}
Assume $u\in H^{1,\tb}_{y^\alpha}(\mathcal{D})$ solves \eqref{02}. Then
\begin{equation}\label{03}
\mathcal{I}[u]=\underset{w\in H^{1,\tb}_{y^\alpha}(\mathcal{D})}{\min}\mathcal{I}[w].
\end{equation}
Conversely, if $u$ satisfies \eqref{03}, then $u$ solves the problem \eqref{02}.
\end{theorem}

\begin{proof}
To prove \eqref{03}, for all $w\in H^{1,\tb}_{y^\alpha}(\mathcal{D})$, $u-w\in H^{1,\tb}_{y^\alpha}(\mathcal{D})$. Then using the fact that $(u-w)|_{\partial_L\mathcal{D}}=(u-w)|_{\Omega\times\{+\infty\}}=0$ and integration by parts we have
\begin{equation}\label{32}
\left(\nabla\cdot\left(y^\alpha \nabla u\right),u-w\right)_{\mathcal{D}}
=-\left(y^\alpha \nabla u,\nabla(u-w)\right)_{\mathcal{D}}+\left(-y^\alpha u_y,u-w\right)_{\Omega\times\{0\}}.
\end{equation}
Note the left hand side is equal to zero since $\nabla\cdot\left(y^\alpha \nabla u\right)\equiv0$. And the second term on the left is
\begin{multline}
\left(-y^\alpha u_y,u-w\right)_{\Omega\times\{0\}}=\left(-y^\alpha u_y|_{y=0},u(x,0)-w(x,0)\right)_{\Omega}\\
=\left(\lim_{y\rightarrow0}(-y^\alpha u_y),u(x,0)-w(x,0)\right)_{\Omega}=\left(d_sf(x),\text{tr}\{u\}-\text{tr}\{w\}\right)_{\Omega}.
\end{multline}
Therefore, it follows \eqref{32}

\begin{equation}
\left(y^\alpha \nabla u,\nabla(u-w)\right)_{\mathcal{D}}=\left(d_sf(x),\ttr\{u\}-\ttr\{w\}\right)_\Omega,
\end{equation}
which implies
\begin{equation}
\left(\nabla u,\nabla u\right)_{y^\alpha,\mathcal{D}}-d_s\left(f,\ttr\{u\}\right)_\Omega=\left(\nabla u,\nabla w\right)_{y^\alpha,\mathcal{D}}-d_s\left(f,\ttr\{w\}\right)_\Omega.
\end{equation}
Using the inequality $\nabla u\cdot\nabla w\leq\frac{1}{2}|\nabla u|^2+\frac{1}{2}|\nabla w|^2$, it leads to
\begin{equation}
\mathcal{I}[u]\leq\mathcal{I}[w].
\end{equation}

On the other hand, suppose \eqref{03} holds. Fix any $w\in H^{1,\tb}_{y^\alpha}(\mathcal{D})$ and write
\begin{equation}
i(\tau):=\mathcal{I}[u+\tau w],\quad \forall\tau\in\mathbb{R}.
\end{equation}
Note
\begin{multline}
i(\tau)=\frac{1}{2}\left(\nabla(u+\tau w),\nabla(u+\tau w)\right)_{y^\alpha,\mathcal{D}}-d_s\left(f,\ttr\{u+\tau w\}\right)_\Omega\\
=\frac{1}{2}\left(\nabla u,\nabla u\right)_{y^\alpha,\mathcal{D}}-d_s\left(f,\ttr\{u\}\right)_\Omega+\tau\left(\left(\nabla u,\nabla w\right)_{y^\alpha,\mathcal{D}}- d_s\left(f,\ttr\{w\}\right)_\Omega\right)+\frac{1}{2}\tau^2\left(\nabla w,\nabla w\right)_{y^\alpha,\mathcal{D}}.
\end{multline}
Since $u+\tau w\in H^{1,\tb}_{y^\alpha}(\mathcal{D})$ for each $\tau$, $i(\tau)$ takes its minimum at $\tau=0$, and thus
\begin{equation}
0=i'(0)=\left(\nabla u,\nabla w\right)_{y^\alpha,\mathcal{D}}- d_s\left(f,\ttr\{w\}\right)_\Omega.
\end{equation}
Using integration by parts we have
\begin{equation}\label{04}
\left(\nabla\cdot\left(y^\alpha \nabla u\right),w\right)_{\mathcal{D}}+\left(d_sf+\underset{y\rightarrow0}{\lim}~y^\alpha u_y,\ttr\{w\}\right)_\Omega=0.
\end{equation}
Especially, \eqref{04} holds for all $w\in C^\infty_c(\mathcal{D})$, which implies
\begin{equation}
\int_{\mathcal{D}}\left(\nabla\cdot\left(y^\alpha \nabla u\right)\right)w=0,\quad\forall w\in C^\infty_c(\mathcal{D}),
\end{equation}
leading to $\nabla\cdot\left(y^\alpha \nabla u\right)=0$ in $\mathcal{D}$. And thus by \eqref{04}
\begin{equation}
\left(d_sf+\underset{y\rightarrow0}{\lim}~y^\alpha u_y,\ttr\{w\}\right)_\Omega=0.
\end{equation}
Especially, $\ttr\{w\}$ takes over all functions in $C^\infty_c(\Omega)$, which leads to $d_sf+\underset{y\rightarrow0}{\lim}~y^\alpha u_y=0$ in $\Omega$.
\end{proof}

By virtue of Theorem \ref{Thm01}, it suffices to solve the following optimization
\begin{equation}\label{13}
\underset{w\in H^{1,\tb}_{y^\alpha}(\mathcal{D})}{\min}\mathcal{I}[w],
\end{equation}
whose solution is exactly a weak solution of the extended problem \eqref{02}.

\section{Neural Network Approximation}
In the numerical computation, one aims to introduce a function set with a finite degree of freedom to approximate the solution space $H^{1,\tb}_{y^\alpha}(\mathcal{D})$, and minimize $\mathcal{I}[\cdot]$ in this appropriate set of functions. In many physical relevant problems, it is required to address $d\geq3$, causing the dimension of $\mathcal{D}$ no less than 4. Potentially high dimensions impede the usage of conventional linear structures. However, as a nonlinear structure, DNNs can approximate high-dimensional functions by moderately less degree of freedom. This inspires us to use classes of DNNs to approximate $H^{1,\tb}_{y^\alpha}(\mathcal{D})$ especially when $d$ is large.

\subsection{Fully connected neural network}\label{Sec_FNN}
In our method, we employ the fully connected neural network (FNN) which is one of the most common neural networks in deep learning. Mathematically speaking, let $\sigma(t)$ be an activation function which is applied entry-wise to a vector $x$ to obtain another vector of the same size. Let $L\in\mathbb{N}^+$ and $M_\ell\in\mathbb{N}^+$ for $\ell=1,\dots,L$, an FNN $\hat{\phi}$ is the composition of $L$ simple nonlinear functions as follows
\begin{equation}
\hat{\phi}(z;\theta):=a^T h_{L} \circ h_{L-1} \circ \cdots \circ h_{1}(z)\quad \text{for } z\in\mathbb{R}^d,
\end{equation}
where $a\in \mathbb{R}^{M_L}$; $h_{\ell}(z_{\ell}):=\sigma\left(W_\ell z_{\ell} + b_\ell \right)$ with $W_\ell \in \mathbb{R}^{M_{\ell}\times M_{\ell-1}}$ ($M_0:=d$) and $b_\ell \in \mathbb{R}^{M_\ell}$ for $\ell=1,\dots,L$. Here $M_\ell$ is called the width of the $\ell$-th layer and $L$ is called the depth of the FNN. $\theta:=\{a,\,W_\ell,\,b_\ell:1\leq \ell\leq L\}$ is the set of all parameters in $\hat{\phi}$ to determine the underlying neural network. Common types of activation functions include the sigmoid function $(1+e^{-t})^{-1}$ and the rectified linear unit (ReLU) $\max(0,t)$. We remark that, when solving $k$-th order differential equations, many existing network-based methods use the ReLU$^{k+1}$ activation function $\max\{0,t^{k+1}\}$, so that their networks are $C^k$ functions and can be applied by the differential operators. While in the minimization \eqref{13}, only $H^1$ regularity is required and therefore ReLU networks suffice.

Denote $M:=\max~\{M_\ell,~1\leq\ell\leq L\}$, then it is clear that $|\theta|=O(M^2L+Md)$. Comparatively, the degree of freedom of linear structures such as finite elements and tensor product polynomials increases exponentially with $d$. Hence FNNs are more practicable in high-dimensional approximations. For simplicity, we consider the architecture $M=M_\ell$ for all $\ell$ and denote $\mathcal{F}_{L,M,\sigma}$ as the set consisting of all FNNs with depth $L$, width $M$ and activation function $\sigma$. In the following passage, all functions involving an FNN will be denoted with the superscript $\hat{}$.

\subsection{Special structures of the approximate class}
Recent work \cite{Lu2020,Shen2021} indicates that deep FNNs can approximate smooth functions in $\infty$-norm within any desired accuracy as long as the depth or width is large enough. The approximation will be more accurate if the target function has higher regularity. However, it is shown in \cite{Capella2011,Chen2020} that the solution of \eqref{02} has a singularity at $y=0$ which behaves as $y^{k+1-\alpha}$ for $k\in\mathbb{N}^+$. Therefore, it is not appropriate to naively use the class of generic FNNs. Instead, We aim to develop a special structure based on FNNs for the approximate class.

In the enriched spectral method \cite{Chen2020}, the solution of \eqref{04} is approximated by a structure consisting of two parts. One part is a linear combination of smooth basis functions, which approximates the regular component of the solution. The other part is a linear tensor product combination of smooth basis functions and a sequence of artificial terms $\{y^{k-\alpha}\}_{k=1,2,\cdots}$, which is for the singular component of the solution. Following this idea, we use $\hat{\phi}(x,y)$ to denote any function in the approximate class, and build its structure as the combination of two parts,
\begin{equation}\label{29}
\hat{\phi}(x,y)=\hat{\phi}_1(x,y)+y^{1-\alpha}\hat{\phi}_2(x,y),
\end{equation}
where $\hat{\phi}_1$, $\hat{\phi}_2$ are FNN-based smooth functions and the term $y^{1-\alpha}$ is introduced to adapt to the singularity at $y=0$.

Moreover, since the true solution of the extended problem \eqref{02} converges to zero as $y\rightarrow\infty$, the functions in the approximate class should also preserve this property. To realize it, we can introduce exponential terms concerning $y$ in the structure of $\hat{\phi}(x,y)$. Specifically, we let
\begin{equation}
\hat{\phi}_1(x,y)=\hat{\phi}_3(x,y)e^{-\gamma'y},\quad\hat{\phi}_2(x,y)=\hat{\phi}_4(x,y)e^{-\gamma''y},
\end{equation}
where $\{\hat{\phi}_3,\hat{\phi}_4\}$ are FNN-based functions and $\{\gamma',\gamma''\}>0$ are two auxiliary scalar parameters ensuring that $\hat{\phi}_1$ and $\hat{\phi}_2$ converges to zero as $y\rightarrow\infty$ exponentially.

In the end, as a subset of $H^{1,\tb}_{y^\alpha}(\mathcal{D})$, the approximate class should also consist of functions satisfying the boundary condition; namely, $\hat{\phi}|_{\partial_L\mathcal{D}}=0$. We achieve this by setting
\begin{equation}
\hat{\phi}_3(x,y)=\hat{\phi}'(x,y)h(x),\quad\hat{\phi}_4(x,y)=\hat{\phi}''(x,y)h(x),
\end{equation}
where $\{\hat{\phi}',\hat{\phi}''\}$ are generic FNNs and $h(x)$ is a smooth function constructed particularly to satisfy $h(x)=0$ on $\partial\Omega$. For example, if $\Omega$ is a hypercube $(a_1,b_1)\times\cdots\times(a_d,b_d)$, then $h(x)$ can be chosen as $h(x)=\prod_{i=1}^d(x_i-a_i)(x_i-b_i)$; if $\Omega$ has a boundary characterized by a level set, say $\partial\Omega=\{x~|~\rho(x)=0\}$, for some continuous function $\rho$, then $h(x)$ can be chosen as $h(x)=\rho(x)$. Generally, we can set $h(x)=F(\text{dist}(x,\partial\Omega))$, where $F$ is a particular analytic function satisfying $F(0)=0$ and $\text{dist}(x,\partial\Omega)$ represents the distant between $x$ and $\partial\Omega$. In practice, we expect to set $h(x)$ as smooth as possible so that $\hat{\phi}_3$ and $\hat{\phi}_4$ preserve the same regularity as $\hat{\phi}'$ and $\hat{\phi}''$, respectively.

To sum up, we build the following special structure for the approximate class,
\begin{equation}\label{05}
\hat{\phi}((x,y);\theta)=\hat{\phi}'((x,y);\theta')h(x)e^{-\gamma'y}+y^{1-\alpha}\hat{\phi}''((x,y);\theta'')h(x)e^{-\gamma''y},
\end{equation}
where $\theta:=\{\theta',\theta'',\gamma',\gamma''\}$ is the set of all trainable parameters. We denote $\mathcal{N}_{L,M,\sigma,h}$ as the class of all neural networks having the structure in \eqref{05}, namely,
\begin{multline}
\mathcal{N}_{L,M,\sigma,h}=\Big\{\hat{\phi}:\mathcal{D}\rightarrow\mathbb{R}~\Big|~\hat{\phi}(x,y)=\hat{\phi}'(x,y)h(x)e^{-\gamma'y}+y^{1-\alpha}\hat{\phi}''(x,y)h(x)e^{-\gamma''y},\\
\forall\hat{\phi}',\hat{\phi}''\in\mathcal{F}_{L,M,\sigma},~\forall\gamma',\gamma''\in\mathbb{R}^+\Big\}.
\end{multline}

It is clear that $\mathcal{N}_{L,M,\sigma,h}$ is a subset of $H^{1,\tb}_{y^\alpha}(\mathcal{D})$ as long as $\sigma$ is Lipschitz continuous and has a polynomial growth bound; namely, $|\sigma(t)|\leq C(1+|t|^p)$ for all $t\in\mathbb{R}$ with a constant $C>0$ and an integer $p>0$. In our Ritz method, $\mathcal{N}_{L,M,\sigma,h}$ is taken as the approximate class of $H^{1,\tb}_{y^\alpha}(\mathcal{D})$.

\subsection{Approximation property}\label{Sec_approximation_property}
We will show that functions in $H^{1,\tb}_{y^\alpha}(\mathcal{D})$ can be approximated by the special networks in $\mathcal{N}_{L,M,\sigma,h}$ as $M\rightarrow\infty$. To illustrate the approximation property, we first introduce the Barron space and then derive the error bounds for neural-network approximation, assuming that the target function has components in the Barron space.
In this section, we specifically focus on the FNNs with ReLU activation and specify $\sigma(x)=\max(x,0)$. For simplicity, we concatenate variables $x\in\mathbb{R}^d$ and $y\in\mathbb{R}$ by writing $z=(x,y)\in\mathbb{R}^{d+1}$.

\subsubsection{Barron space}
Let us first quickly review the Barron space and norm. We will focus on the definition discussed in \cite{E2020} which represents infinitely wide two-layer ReLU FNNs. Following Section \ref{Sec_FNN}, recall the set of two-layer ReLU FNNs without output bias is given by
\begin{equation}
\mathcal{F}_{2,M,\text{ReLU}}=\left\{\hat{\phi}~\Big|~\hat{\phi}(z)=\frac{1}{M}\sum_{i=1}^Ma_i\sigma(b_i^Tz+c_i),\quad\forall(a_i,b_i,c_i)\in \mathbb{R}\times\mathbb{R}^{d+1}\times\mathbb{R}\right\}.
\end{equation}
For a probability measure $\pi$ on $\mathbb{R}\times\mathbb{R}^{d+1}\times\mathbb{R}$, we set the function
\begin{equation}
f_\pi(z)=\int_\mathcal{D}a\sigma(b^Tz+c)\pi(\td a,\td b,\td c)=\mathbb{E}_\pi[a\sigma(b^Tz+c)],\quad\forall z\in\mathbb{R}^{d+1},
\end{equation}
given this expression exists. For a function $u:\mathbb{R}^{d+1}\rightarrow\mathbb{R}$, we use $\Pi_u$ to denote the set of all probability measures $\pi$ such that $f_\pi(z)=u(z)$ almost everywhere. Then the Barron norm is defined as
\begin{equation}
\|u\|_\mathcal{B}^2:=\underset{\pi\in\Pi_u}{\inf}\int_{\mathbb{R}\times\mathbb{R}^{d+1}\times\mathbb{R}}a^2(|b|+|c|)^2\pi(\td a,\td b,\td c)=\underset{\pi\in\Pi_u}{\inf}~\mathbb{E}_\pi[a^2(|b|+|c|)^2].
\end{equation}
The infimum of the empty set is considered as $+\infty$. The set of all functions with finite Barron norm is denoted by $\mathcal{B}$. It is shown in \cite{E2020} that $\mathcal{B}$ equipped with the Barron norm is a Banach space which is called Barron space.

\subsection{Error estimation}
Let $u$ be a function in $\in H^{1,\tb}_{y^\alpha}(\mathcal{D})$, and further assume that $u$ converges to zero exponentially as $y\rightarrow\infty$. In the error analysis, we make the following assumption that $u$ can be factorized explicitly by components vanishing on $\partial_L\mathcal{D}$ and decaying to zero exponentially as $y\rightarrow\infty$.
\begin{assum}\label{Assum01}
There exist functions $h\in H_0^1(\Omega)$, $v\in\mathcal{B}$ and some number $\eta>0$ such that
\begin{equation}\label{14}
u(z)=v(z)h(x)e^{-\eta y}.
\end{equation}
\end{assum}
Especially, if Assumption \ref{Assum01} holds, we can normalize $h$ such that $|h|\leq1$ and $|\nabla h|\leq1$. Assumption \ref{Assum01} is indeed satisfied in some situations. For example, in the case $\Omega=[-1,1]^2$, $s=0.5$ and $f=\sqrt{2}\pi\sin(\pi x_1)\sin(\pi x_2)$, the solution $u$ of \eqref{02} is given by $u=\sin(\pi x_1)\sin(\pi x_2)e^{-\sqrt{2}\pi y}$. Note that by Taylor series $\sin(\pi x)=-\frac{\pi}{2}(x^2-1)+\frac{\pi}{8}(x^2-1)^2+O\left((x^2-1)^3\right)$, $u$ can be written as \eqref{14} with
\begin{gather*}
v=\left(-\pi+\frac{\pi}{4}(x_1^2-1)+O\left((x_1^1-1)^2\right)\right)\left(-\pi+\frac{\pi}{4}(x_2^2-1)+O\left((x_2^2-1)^2\right)\right),\\
h=\frac{(x_1^2-1)(x_2^2-1)}{4},\quad\eta=\sqrt{2}\pi.
\end{gather*}

Now we investigate the approximation error between $\mathcal{N}_{2,M,\text{ReLU},h}$ and $H^{1,\tb}_{y^\alpha}(\mathcal{D})$. It suffices to consider a special subset $\mathcal{S}_{2,M,\text{ReLU},h}$ given by
\begin{equation}
\mathcal{S}_{2,M,\text{ReLU},h}=\left\{\hat{u}~|~\hat{u}(z)=\hat{v}(z)h(x)e^{-\eta y},\quad\forall\hat{v}\in\mathcal{F}_{2,M,\text{ReLU}}\right\}.
\end{equation}
Note in $\mathcal{S}_{2,M,\text{ReLU},h}$ only the parameters of $\hat{v}$ are free and trainable, while $\eta$ is fixed. Clearly, $\mathcal{S}_{2,M,\text{ReLU},h}$ is a subset of $\mathcal{N}_{2,M,\text{ReLU},h}$. The following theorem and proof are referred to the result in \cite{Muller2021} for deep Ritz method in a bounded domain.

\begin{theorem}\label{Thm02}
Let $u\in H^{1,\tb}_{y^\alpha}(\mathcal{D})$. If Assumption \ref{Assum01} is true with $|h|\leq1$ and $|\nabla h|\leq1$, then there exists some $\hat{u}\in\mathcal{S}_{2,M,\text{ReLU},h}$ such that
\begin{equation}\label{22}
\|\hat{u}-u\|_{H^{1,\tb}_{y^\alpha}(\mathcal{D})}\leq C(\Omega,\eta)M^{-\frac{1}{2}}\|v\|_\mathcal{B},
\end{equation}
with
\begin{equation}\label{24}
C(\Omega,\eta)=2^\frac{3}{2}|\Omega|^\frac{1}{2}\left[\frac{R_\Omega^2+1}{\alpha+1}+\frac{1}{\alpha+3}+\frac{(4\eta^3+2\eta^2)(R_\Omega^2+1)+4\eta^3+6\eta^2+6\eta+3}{8\eta^4}e^{-2\eta}\right]^\frac{1}{2},
\end{equation}
where $R_\Omega:=\max\left(\sup_{x\in\Omega}|x|,1\right)$.
\end{theorem}

\begin{proof}
By the definition of the Barron norm, there exsits some probability measure $\pi$ such that $f_\pi=v$ a.e. and $\mathbb{E}_\pi\left[a^2(|b|+|c|)^2\right] \leq 2\|v\|_\mathcal{B}^2$. For all $(a,b,c)\in\mathbb{R}\times\mathbb{R}^{d+1}\times\mathbb{R}$, using the facts $|\sigma(t)|\leq|t|$ and $|\sigma'(t)|=\chi_{t\geq0}$ we have
\begin{eqnarray}\label{15}
\|a\sigma(b^Tz+c)\|_{1,e^{-2\eta y}y^\alpha,\mathcal{D}}^2 & = &
\int_\mathcal{D}\left[|a\sigma(b^Tz+c)|^2+|\nabla(a\sigma(b^Tz+c))|^2\right]e^{-2\eta y}y^\alpha\td z \nonumber \\
& \leq & \int_\mathcal{D}\left[a^2(b^Tz+c)^2+a^2\chi_{b^Tz+c\geq0}|\nabla(b^Tz+c)|^2\right]e^{-2\eta y}y^\alpha\td z \nonumber \\
& \leq & \int_\mathcal{D}\left[a^2(|b||z|+|c|)^2+a^2|b|^2\right]e^{-2\eta y}y^\alpha\td z.
\end{eqnarray}
Since $|z|=(|x|^2+y^2)^\frac{1}{2}\leq(R_\Omega^2+y^2)^\frac{1}{2}$, we have
\begin{eqnarray}\label{16}
& & \int_\mathcal{D}\left[a^2(|b||z|+|c|)^2+a^2|b|^2\right]e^{-2\eta y}y^\alpha\td z
\nonumber \\
& \leq & \int_\mathcal{D}\left[(R_\Omega^2+y^2)a^2(|b|+|c|)^2+a^2|b|^2\right]e^{-2\eta y}y^\alpha\td z \nonumber \\
& \leq & a^2(|b|+|c|)^2\int_\mathcal{D}(R_\Omega^2+y^2+1)e^{-2\eta y}y^\alpha\td z
\nonumber \\
& \leq & a^2(|b|+|c|)^2|\Omega|\int_0^\infty(R_\Omega^2+y^2+1)e^{-2\eta y}y^\alpha\td y:=a^2(|b|+|c|)^2|\Omega|\cdot I_1.
\end{eqnarray}
While $I_1$ is bounded above since
\begin{multline}\label{17}
I_1\leq\int_0^1(R_\Omega^2+y^2+1)y^\alpha\td y+\int_1^\infty(R_\Omega^2+y^2+1)e^{-2\eta y}y\td y\\
=\left(\frac{R_\Omega^2+1}{\alpha+1}+\frac{1}{\alpha+3}\right)+\frac{(4\eta^3+2\eta^2)(R_\Omega^2+1)+4\eta^3+6\eta^2+6\eta+3}{8\eta^4}e^{-2\eta}:=I_2.
\end{multline}
Combining \eqref{15},\eqref{16},\eqref{17} we have
\begin{equation}\label{18}
\mathbb{E}_\pi\left[\|a\sigma(b^Tz+c)\|_{1,e^{-2\eta y}y^\alpha,\mathcal{D}}^2\right]\leq \mathbb{E}_\pi\left[a^2(|b|+|c|)^2\right]|\Omega|\cdot I_2\leq2I_2|\Omega|\cdot\|v\|_\mathcal{B}^2.
\end{equation}

On the other hand, the mapping
\begin{equation*}
\mathbb{R}\times\mathbb{R}^{d+1}\times\mathbb{R},\quad (a,b,c)\rightarrow a\sigma(b^Tz+c)
\end{equation*}
is continuous and hence Bochner measurable. Also, \eqref{18} leads to
\begin{equation*}
\mathbb{E}_\pi\left[\|a\sigma(b^Tz+c)\|_{1,e^{-2\eta y}y^\alpha,\mathcal{D}}\right]\leq\left(\mathbb{E}_\pi\left[\|a\sigma(b^Tz+c)\|_{1,e^{-2\eta y}y^\alpha,\mathcal{D}}^2\right]\right)^\frac{1}{2}<\infty,
\end{equation*}
which implies the integral $\int_\mathcal{D}a\sigma(b^Tz+c)\pi(\td a,\td b,\td c)$ is a Bochner integral.

We note the fact that if $\xi_1,\cdots,\xi_M$ are independent samples from a random variable $\xi$, then
\begin{eqnarray*}
& &
\mathbb{E}\left(\frac{1}{M}\sum_{i=1}^M\xi_i-\mathbb{E}\xi\right)^2=\mathbb{E}\left(\frac{1}{M}\sum_{i=1}^M(\xi_i-\mathbb{E}\xi)\right)^2 \\
& = & \frac{1}{M^2}\left(\sum_{i=1}^M\mathbb{E}(\xi_i-\mathbb{E}\xi)^2+\sum_{1\leq i<j\leq M}\mathbb{E}(\xi_i-\mathbb{E}\xi)\cdot\mathbb{E}(\xi_j-\mathbb{E}\xi)\right)=\frac{1}{M}\sum_{i=1}^M\mathbb{E}(\xi-\mathbb{E}\xi)^2\\
& =& \frac{1}{M}\mathbb{E}\xi^2-\frac{1}{M}(\mathbb{E}\xi)^2\leq\frac{1}{M}\mathbb{E}\xi^2.
\end{eqnarray*}
By similar argument, for independent samples $\left\{(a_i,b_i,c_i)\right\}$ from $\pi$, we have
\begin{equation*}
\mathbb{E}_{\pi^M}\left[\left\|\frac{1}{M}\sum_{i=1}^M a_i\sigma(b_i^Tz+c_i)-f_\pi(z)\right\|_{1,e^{-2\eta y}y^\alpha,\mathcal{D}}^2\right]\leq\frac{1}{M}\mathbb{E}_\pi\left[\|a\sigma(b^Tz+c)\|_{1,e^{-2\eta y}y^\alpha,\mathcal{D}}^2\right].
\end{equation*}
In particular, there exists $\left\{(a_i,b_i,c_i)\right\}$ such that
\begin{equation}\label{19}
\left\|\frac{1}{M}\sum_{i=1}^M a_i\sigma(b_i^Tz+c_i)-f_\pi(z)\right\|_{1,e^{-2\eta y}y^\alpha,\mathcal{D}}^2\leq\frac{1}{M}\mathbb{E}_\pi\left[\|a\sigma(b^Tz+c)\|_{1,e^{-2\eta y}y^\alpha,\mathcal{D}}^2\right].
\end{equation}
Let $\hat{v}=\frac{1}{M}\sum_{i=1}^M a_i\sigma(b_i^Tz+c_i)$, combining \eqref{18} and \eqref{19} have obtain
\begin{equation}\label{20}
\left\|\hat{v}-v\right\|_{1,e^{-2\eta y}y^\alpha,\mathcal{D}}^2\leq 2I_2|\Omega|M^{-1}\cdot\|v\|_\mathcal{B}^2
\end{equation}

Finally, let $\hat{u}=\hat{v}(z)h(x)e^{-\eta y}$ and denote $\bar{v}:=\hat{v}-v$, then
\begin{eqnarray}\label{21}
& & \|\hat{u}-u\|_{H^{1,\tb}_{y^\alpha}(\mathcal{D})}^2\nonumber \\
&=& \|\nabla(\bar{v}(z)h(x)e^{-\eta y})\|_{y^\alpha,\mathcal{D}}^2 \nonumber \\
& = & \int_\mathcal{D}\left(|h(x)\nabla_x\bar{v}(z)+\bar{v}(z)\nabla_xh(x)|^2+|\partial_y\bar{v}(z)-\bar{v}(z)|^2h^2(x)\right)e^{-2\eta y}y^\alpha\td z \nonumber \\
& \leq& 4\int_\mathcal{D}\left(|\nabla \bar{v}|^2+|\bar{v}|^2\right)e^{-2\eta y}y^\alpha\td z=4\|\hat{v}-v\|_{1,e^{-2\eta y}y^\alpha,\mathcal{D}}^2.
\end{eqnarray}
The inequality \eqref{22} directly follows \eqref{20} and \eqref{21}.
\end{proof}

Theorem \ref{Thm02} provides a fractional equation dimension-independent approximation error bound for neural networks in $\mathcal{S}_{2,M,\text{ReLU},h}$ given that the target function $u$ has a Barron component. This is a desired property since it avoids the curse of dimensionality. Since $\mathcal{S}_{2,M,\text{ReLU},h}\subset\mathcal{N}_{2,M,\text{ReLU},h}$, Theorem \ref{Thm02} also holds for $\mathcal{N}_{2,M,\text{ReLU},h}$.

\section{Ritz method and Error Estimation}
The extended problem \eqref{02} can be solved practically by Ritz method. Thanks to the approximation property of the neural network class discussed in Section \ref{Sec_approximation_property}, we can replace the hypothesis space $H^{1,\tb}_{y^\alpha}(\mathcal{D})$ with $\mathcal{N}_{L,M,\sigma,h}$ in the minimization \eqref{13}, obtaining
\begin{equation}\label{23}
\underset{\hat{w}\in \mathcal{N}_{L,M,\sigma,h}}{\min}\mathcal{I}[\hat{w}].
\end{equation}
Then the solution of \eqref{23} will be an approximation to the solution of \eqref{13}. Same as in Section \ref{Sec_approximation_property}, we will estimate the solution error for the two-layer ReLU networks; namely, we consider the case that $L=2$ and $\sigma(t)=\max(0,t)$. The final error of the original problem \eqref{01} will thereafter be presented.

\subsection{Error estimation}
Let $u^*\in H^{1,\tb}_{y^\alpha}(\mathcal{D})$ be a minimizer of \eqref{13}; namely,
\begin{equation}
\mathcal{I}[u^*]=\min_{w\in H^{1,\tb}_{y^\alpha}(\mathcal{D})}\mathcal{I}[w]
\end{equation}
Given that Assumption \ref{Assum01} holds for $u^*$ with a factorization $u^*(z)=v^*(z)h^*(x)e^{-\eta^* y}$, let $\hat{u}^*\in \mathcal{N}_{2,M,\text{ReLU},h^*}$ be a minimizer of \eqref{23}; namely,
\begin{equation}
\mathcal{I}[\hat{u}^*]=\min_{\hat{w}\in \mathcal{N}_{2,M,\text{ReLU},h^*}}\mathcal{I}[\hat{w}].
\end{equation}

We first introduce the following C\'{e}a Lemma \cite{Miller2001}.
\begin{proposition}\label{Prop01}
Let $X$ be a Hilbert space, $V\subset X$ any subset and $a:X\times X\rightarrow\mathbb{R}$ a symmetric, continuous and $\alpha$-coercive bilinear form. For $f\in X'$ define the quadratic energy $E_f(u):=\frac{1}{2}a(u,u)-f(u)$ and denote its unique minimizer by $u_f$. Then for every $v\in V$ it holds that
\begin{equation}
\|v-u_f\|_X\leq\sqrt{\alpha^{-1}\left(2\left(E_f(v)-\inf_{\tilde{v}\in V}E_f(\tilde{v})\right)+\inf_{\tilde{v}\in V}(\tilde{v}-u_f,\tilde{v}-u_f)\right)}.
\end{equation}
\end{proposition}
It is trivial to show that $(\nabla\cdot,\nabla\cdot)_{y^\alpha,\mathcal{D}}$ is a symmetric, continuous and $1$-coercive bilinear form. Therefore by Lemma \ref{Prop01} and Theorem \ref{Thm02} we have
\begin{multline}\label{27}
\|\hat{u}^*-u^*\|_{H^{1,\tb}_{y^\alpha}(\mathcal{D})}\leq\sqrt{\inf_{\hat{u}\in \mathcal{N}_{2,M,\text{ReLU},h^*}}(\nabla(\hat{u}-u^*),\nabla(\hat{u}-u^*))_{y^\alpha,\mathcal{D}}}\\
\leq\inf_{\hat{u}\in \mathcal{N}_{2,M,\text{ReLU},h^*}}\|\hat{u}-u^*\|_{H^{1,\tb}_{y^\alpha}(\mathcal{D})}\leq C(\Omega,\eta^*)M^{-\frac{1}{2}}\|v^*\|_\mathcal{B}.
\end{multline}

To derive the error for the original problem \eqref{01}, we introduce the following trace theorem \cite{Nochetto2015},
\begin{proposition}\label{Prop02}
The trace operator $\ttr$ defined in \eqref{25} satisfies $\ttr\{H^{1,\tb}_{y^\alpha}(\mathcal{D})\}=\mathbb{H}^s(\Omega)$, and
\begin{equation}\label{28}
\|\ttr\{u\}\|_{\mathbb{H}^s(\Omega)}\leq C\|u\|_{H^{1,\tb}_{y^\alpha}(\mathcal{D})},\quad\forall u\in H^{1,\tb}_{y^\alpha}(\mathcal{D}),
\end{equation}
where $C$ is a constant only depending on $\Omega$ and $s$, and the space $\mathbb{H}^s(\Omega)$ is defined as
\begin{equation}
\mathbb{H}^s(\Omega)=\left\{u=\sum_{n=1}^\infty\tilde{u}_n\phi_n\in L^2(\Omega):\|u\|_{\mathbb{H}^s}=\sum_{n=1}^\infty\lambda_n^s|\tilde{u}_n|^2<\infty\right\},
\end{equation}
with $(\lambda_n,\phi_n)$ being all the eigenvalues and orthonormal eigenfunctions of \eqref{26}. Especially, $\mathbb{H}^s(\Omega)$ can characterized by
\begin{equation}
\mathbb{H}^s(\Omega)u=\begin{cases}
H^s(\Omega),\quad s\in(0,\frac{1}{2}),\\H_{00}^\frac{1}{2}(\Omega),\quad s=\frac{1}{2},\\H_0^s(\Omega),\quad s\in(\frac{1}{2},1),\end{cases}
\end{equation}
where
\begin{equation}
H_{00}^\frac{1}{2}(\Omega)=\left\{u\in H^\frac{1}{2}(\Omega):\int_\Omega\frac{|u(x)|^2}{\text{\rm dist}(x,\partial\Omega)}\td x<\infty\right\}.
\end{equation}
\end{proposition}

Recall that the solution of original problem \eqref{01} is exactly the trace of the solution of the extended problem \eqref{02}. Also, Theorem \ref{Thm01} shows the equivalence between the extended problem \eqref{02} and the minimization problem \eqref{13}. Therefore, combining \eqref{27} and \eqref{28} leads to the following error estimate for the original solution.
\begin{theorem}\label{Thm03}
Given that the solution $u^*$ of \eqref{13} satisfies Assumption \ref{Assum01}, say $u^*(z)=v^*(z)h^*(x)e^{-\eta^* y}$. Let $U^*(x)$ be the solution of \eqref{01} and $\hat{u}^*(x,y)$ be the solution of \eqref{23} with $L=2$, $h=h^*$ and $\sigma(\cdot)=\max(0,\cdot)$. Then
\end{theorem}
\begin{equation}
\|\hat{u}^*(x,0)-U^*(x)\|_{\mathbb{H}^s(\Omega)}\leq C M^{-\frac{1}{2}}\|v^*\|_\mathcal{B},
\end{equation}
where $C$ is a constant only depending on $\Omega$, $s$ and $\eta^*$.

For our network-based method, we obtain the solution error $O(M^{-1/2})$ provided that the true solution has a Barron component. This order $-1/2$ is consistent with that of the deep Ritz method solving regular Laplacian equations \cite{Muller2021} proved under the similar Barron framework. Moreover, the proposed method can be compared with existing methods solving fractional Laplacian \cite{Nochetto2015,Chen2020}. In the early work, finite elements or Laguerre functions are used to approximate the solution, having error orders of $O(N_\text{p}^{-\nu/d})$, where $\nu$ characterizes the regularity of the true solution and $N_\text{p}$ is the number of free parameters. If $d$ is relatively low, say $d<2\nu$, then $-\nu/d<-1/2$ and hence the early methods converges faster. Otherwise, the network-based methods outperform the existing methods in the error order. This fact implies that neural networks are advantageous over other classical structures in the high-dimensional approximation.

\subsection{Implementation}
In the proposed method, we solve the optimization \eqref{23}, finding a minimizer of $\mathcal{I}[\cdot]$ in the hypothesis space $\mathcal{N}_{L,M,\sigma,h}$. Note that
\begin{equation}\label{12}
\mathcal{I}[\hat{\phi}(x,y)]=\frac{1}{2}\int_0^\infty y^\alpha\left(\int_\Omega|\nabla \hat{\phi}(x,y)|^2\td x\right)\td y-d_s\int_\Omega f(x)\hat{\phi}(x,0)\td x.
\end{equation}
In practice, numerical quadrature is required for the integrals in \eqref{12}. Since $d$ might be moderately large, for the integral in terms of $x$ over $\Omega$, one choice is using Monte Carlo-type quadrature. Specifically, we prescribe a set of $N$ quadrature nodes $\mathcal{T}=\{x_n\}_{n=1}^N$ which are uniformly distributed in $\Omega$, then for each $y\in(0,\infty)$,
\begin{equation}\label{07}
\int_\Omega|\nabla \hat{\phi}(x,y)|^2\td x\approx\frac{|\Omega|}{N}\sum_{n=1}^N|\nabla\hat{\phi}(x_n,y)|^2,\quad\int_\Omega f(x)\hat{\phi}(x,0)\td x\approx\frac{|\Omega|}{N}\sum_{n=1}^Nf(x_n)\hat{\phi}(x_n,0).
\end{equation}
For the integral in terms of $y$ over $(0,\infty)$, specific quadrature rule is needed due to the singularity at $y=0$. It is shown in \cite{Bonito2017,Bonito2019} that the infinite integral involving fractional operators can be effectively computed by sinc quadrature. More precisely, we use the change of variable $y=e^\mu$ so that
\begin{equation}\label{06}
\int_0^\infty y^\alpha g(y)\td y=\int_{-\infty}^\infty e^{(\alpha+1)\mu}g(e^\mu)\td\mu
\end{equation}
for all $g\in L^1_{y\alpha}(0,\infty)$. Given $\bar{h}>0$, let
\begin{equation}
N_+:=\lceil\frac{\pi^2}{4s\bar{h}^2}\rceil,\quad N_-:=\lceil\frac{\pi^2}{4(1-s)\bar{h}^2}\rceil,\quad\mu_m:=m\bar{h},
\end{equation}
then \eqref{06} is evaluated by quadrature nodes $\{\mu_m\}_{m=-N_-,\cdots,N_+}$ and uniform weights $\bar{h}$, namely,
\begin{equation}\label{08}
\int_{-\infty}^\infty e^{(\alpha+1)\mu}g(e^\mu)\td\mu\approx\bar{h}\sum_{m=-N_-}^{N_+}e^{(\alpha+1)\mu_m}g(e^{\mu_m}).
\end{equation}

Combining \eqref{07} and \eqref{08} we have an approximate functional of $\mathcal{I}$ given by
\begin{equation}\label{10}
\mathcal{I}_{\mathcal{T},\bar{h}}[\hat{\phi}(x,y)]=\frac{|\Omega|}{N}\left(\bar{h}\sum_{m=-N_-}^{N_+}e^{(\alpha+1)\mu_m}\sum_{n=1}^N|\nabla\hat{\phi}(x_n,e^{\mu_m})|^2-d_s\sum_{n=1}^Nf(x_n)\hat{\phi}(x_n,0)\right).
\end{equation}

Practically, we solve the following optimization,
\begin{equation}\label{09}
\underset{\hat{\phi}\in\mathcal{N}_{L,M,\sigma,h}}{\min}\mathcal{I}_{\mathcal{T},\bar{h}}[\hat{\phi}],
\end{equation}
whose solution can be regarded as a practical numerical solution of \eqref{13}.

\section{Numerical Experiments}
\subsection{The setting}
The proposed method is tested by numerical experiments in this section. Deep learning techniques are utilized to solve the minimization \eqref{09}. The overall setting is summarized as follows.
\begin{itemize}
  \item {\em Environment.}
  The experiment is performed in Python environment. PyTorch library with CUDA toolkit is utilized for neural network implementation and GPU-based parallel computing. The codes can be simply implemented on a desktop.
  \item {\em Optimizer and hyper-parameters.}
  The network-based optimization \eqref{09} is solved by the stochastic gradient descent (SGD) from PyTorch library. The SGD is implemented for totally 5000 epochs, with adaptively decaying learning rates.
  \item {\em Network setting.}
  The special network structure \eqref{05} is adopted. We choose $\sigma$ as the ReLU function. The parameters in $\theta'$ and $\theta''$ are initialized by
  \begin{equation}
  a, W_l, b_l\sim U(-\sqrt{M},\sqrt{M}),\quad l=1,\cdots,L.
  \end{equation}
  And $\gamma',\gamma''$ are initialized as 0.5. All these parameters are trained in the learning process.
  \item {\em Numerical quadrature.}
  For quadrature \eqref{07} over $\Omega$, we adopt the quasi-Monte Carlo with Halton sequence. For the cases of $d=3$ and $10$, totally $N=10^5$ and $5\times10^5$ sample points are used in the Monte Carlo quadrature, separated as 4 and 10 batches in the SGD, respectively. The sinc quadrature parameter $\bar{h}=$ is set as $1/3$.
  \item {\em Testing set and error evaluation.}
  We generate a testing set $\mathcal{X}$ consisting of $10^4$ random points uniformly distributed in $\Omega$ for error evaluation. Suppose $\hat{\phi}(x,y)$ is the neural network obtained by our method and $U(x)$ is the true solution of the original fractional Laplacian problem \eqref{01}, then the following relative $\ell^2$ error will be computed,
  \begin{equation}
  e_{\ell^2}(\mathcal{X}):=\left(  \sum_{x\in\mathcal{X}}|\hat{\phi}(x,0)-U(x)|^2 / \sum_{x\in\mathcal{X}}|U(x)|^2 \right)^\frac{1}{2}.
  \end{equation}
\end{itemize}

\subsection{A model problem}
In this example, we solve the following problem with an explicit solution to test the accuracy of the proposed method,
\begin{equation}\label{case1}
\begin{cases}
(-\Delta)^s u(x) = (d\pi^2)^s\prod_{i=1}^d\sin(\pi x_i),\quad\forall x:=[x_1,\cdots,x_d]\in\Omega:=(-1,1)^d,\\
u(x) = 0,\quad \forall x\in\partial\Omega,
\end{cases}
\end{equation}
whose true solution is given by $u(x)=\prod_{i=1}^d\sin(\pi x_i)$. We specify the boundary function $h(x)=\prod_{i=1}^d(1-x_i^2)$ in the network architecture \eqref{05}.

\subsubsection{Network size test}
First, we solve the problem \eqref{case1} with $s=0.5$ using special network \eqref{05} of various network sizes. The proposed method is implemented with network depth $L=2,3$ and width $M=25,50,100,200$. The errors of the numerical solutions are listed in Table \ref{Tab_case1_size_errors_d_3} as well as the total running time. We also computed the numerical error orders with respect to $M$. Note that most of the running time is occupied by training the networks, and very little is for the testing process (computing the errors).

From the table, it is observed that using larger sizes improves the accuracy at the expense of extra running time. For a fixed width $M$, the obtained error is reduced by more than half from $L=2$ to $L=3$. For a fixed depth $L$, the numerical order with respect to $M$ is roughly around the theoretical order $-1/2$ proved in Theorem \ref{Thm03}, and the deviation is probably due to the stochasticity and capability of the learning algorithm. Overall, the size pair $(L,M)=(3,200)$ obtains the most accurate solution; hence we continue using it in the following tests.

\begin{table}
\centering
\begin{tabular}{|c|c|c|c|c|c|c|}
  \hline
   & \multicolumn{3}{c|}{$L=2$} & \multicolumn{3}{c|}{$L=3$} \\\hline
  $M$ & Error & Order & Running Time & Error & Order & Running Time\\\hline
  $25$ & $8.252\times10^{-2}$ & N.A. & $7.8\times10^3$ & $2.410\times10^{-2}$ & N.A. & $8.0\times10^3$\\\hline
  $50$ & $4.337\times10^{-2}$ & $-0.93$ & $7.7\times10^3$ & $2.071\times10^{-2}$ & $-0.22$ & $8.0\times10^3$\\\hline
  $100$ & $3.522\times10^{-2}$ & $-0.3$ & $7.9\times10^3$ & $1.371\times10^{-2}$ & $-0.59$ & $8.6\times10^3$\\\hline
  $200$ & $2.529\times10^{-2}$ & $-0.48$ & $8.6\times10^3$ & $8.424\times10^{-3}$ & $-0.7$ & $1.3\times10^4$\\\hline
\end{tabular}
\caption{\em Errors $e_{\ell^2}(\mathcal{X})$, error orders with respect to $M$ and running time (seconds) for various $L$ and $M$ when $d=3$.}
\label{Tab_case1_size_errors_d_3}
\end{table}

\subsubsection{Comparison of structures}
Next, this problem is solved for $d=3$ with $s=0.1$, $0.3$, $0.5$, $0.7$, $0.9$. We test both the proposed special structure \eqref{05} and the following simple FNN structure for a comparison,
\begin{equation}\label{30}
\hat{\phi}(x,y)=\hat{\phi}'(x,y)h(x)e^{-\frac{1}{2}y},
\end{equation}
where $\hat{\phi}'$ is a generic FNN. Note in \eqref{30}, the power of the exponent is fixed with $-1/2$ instead of a trainable parameter, and no singular term is introduced. By this design, the comparison can reveal the advantages of using trainable exponential decaying powers and the special singular term. For both structures, we set $L=3$ and $M=200$ for the involved FNN.

The error curves versus epochs of the SGD are shown in Figure \ref{Fig_case1_errors_d_3}, and the errors of the finally obtained solutions are listed in Table \ref{Tab_case1_errors_d_3}. It is observed that for each structure, the smallest error is obtained when $s=0.5$, while larger errors are obtained when $s$ is close to 0 or 1. This is natural since the true solution has no singularity if $s=0.5$ and has higher singularity if $s$ approaches to 0 or 1. Moreover, from the comparison, it is clear that the special structure outperforms the simple one in obtaining smaller errors. We also remark that the time cost of the simple structure is only slightly less than the special structure.

\subsubsection{High-dimensional simulation}
Finally, we conduct a high-dimensional test by solving the model problem \eqref{case1} of $d=10$ and $s=0.2,0.3,0.4,0.5,0.6,0.7,0.8$. We continue using the special structure of $L=3$ and $M=200$ for approximation. The error curves and final errors are shown in Figure \ref{Fig_case1_errors_d_10} and Table \ref{Tab_case1_errors_d_10}. We observe that the proposed method is still effective for high dimensions, obtaining the errors $O(10^{-2})$. While the smallest error is obtained between $s=0.4$ and $0.5$.

To the best of our knowledge, our method is the first successful attempt at 10-D spectral fractional Laplacian equation. The existing methods \cite{Nochetto2015,Chen2020} using finite elements or Laguerre functions show their effectiveness in low-dimensional problems ($d=1,2$), while they are incapable of high-dimensional cases due to the limitation of storage. Therefore, we do not conduct a numerical comparison in this work.

\begin{figure}
\centering
\subfloat[$s=0.1$]{
\includegraphics[scale=0.5]{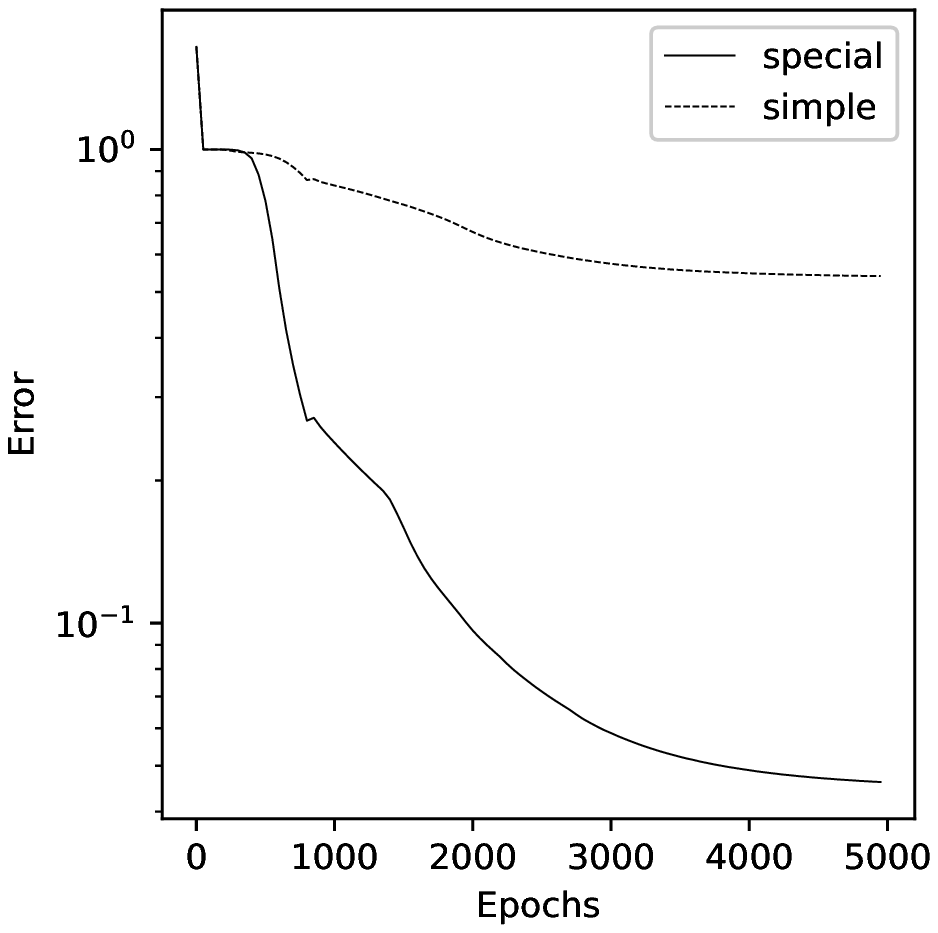}}
\subfloat[$s=0.3$]{
\includegraphics[scale=0.5]{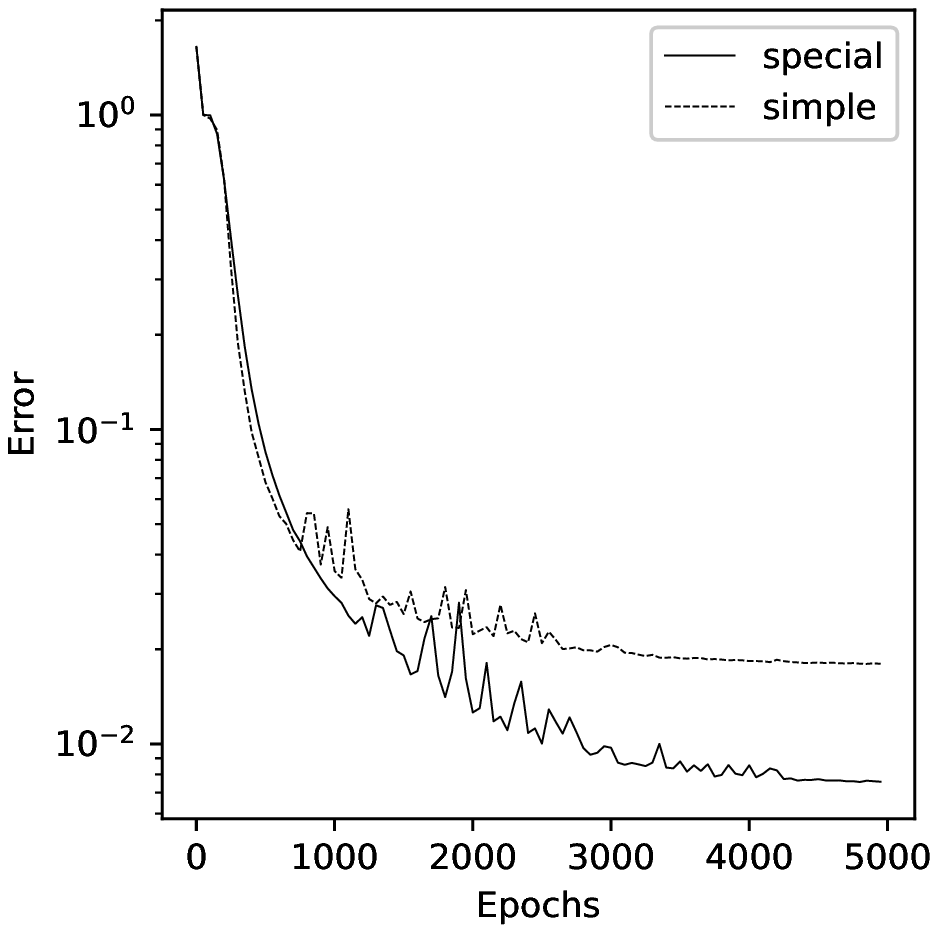}}\\
\subfloat[$s=0.5$]{
\includegraphics[scale=0.5]{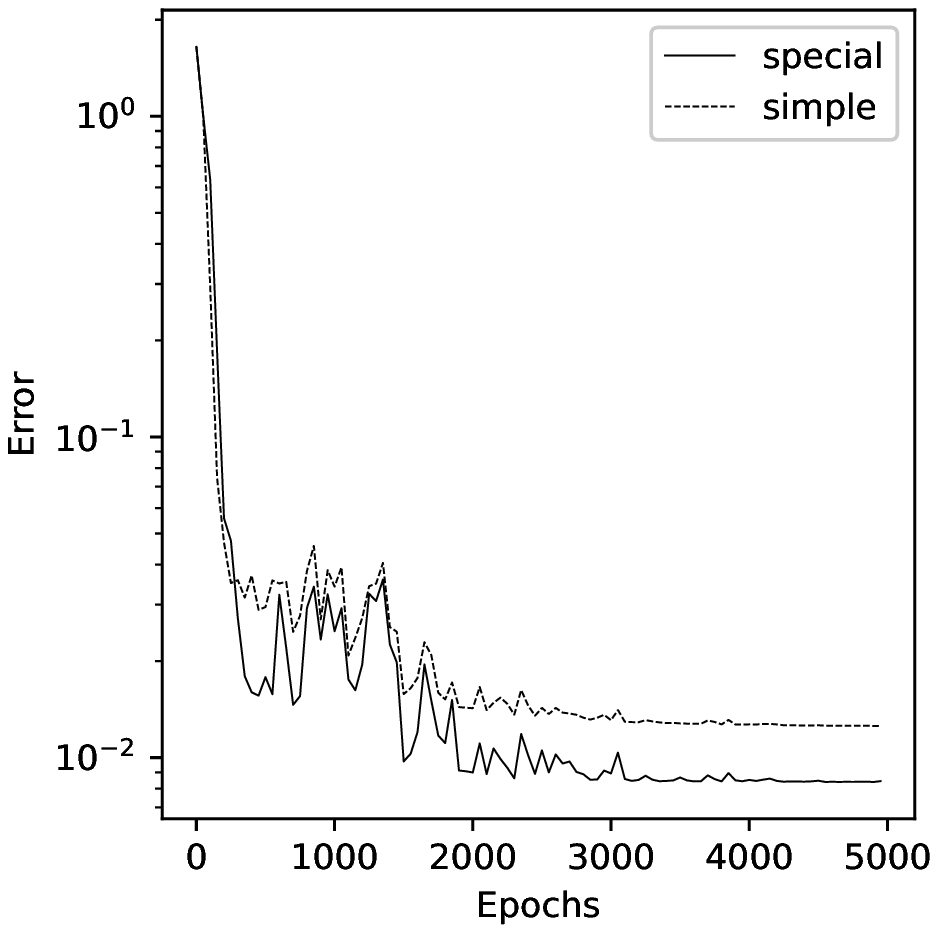}}
\subfloat[$s=0.7$]{
\includegraphics[scale=0.5]{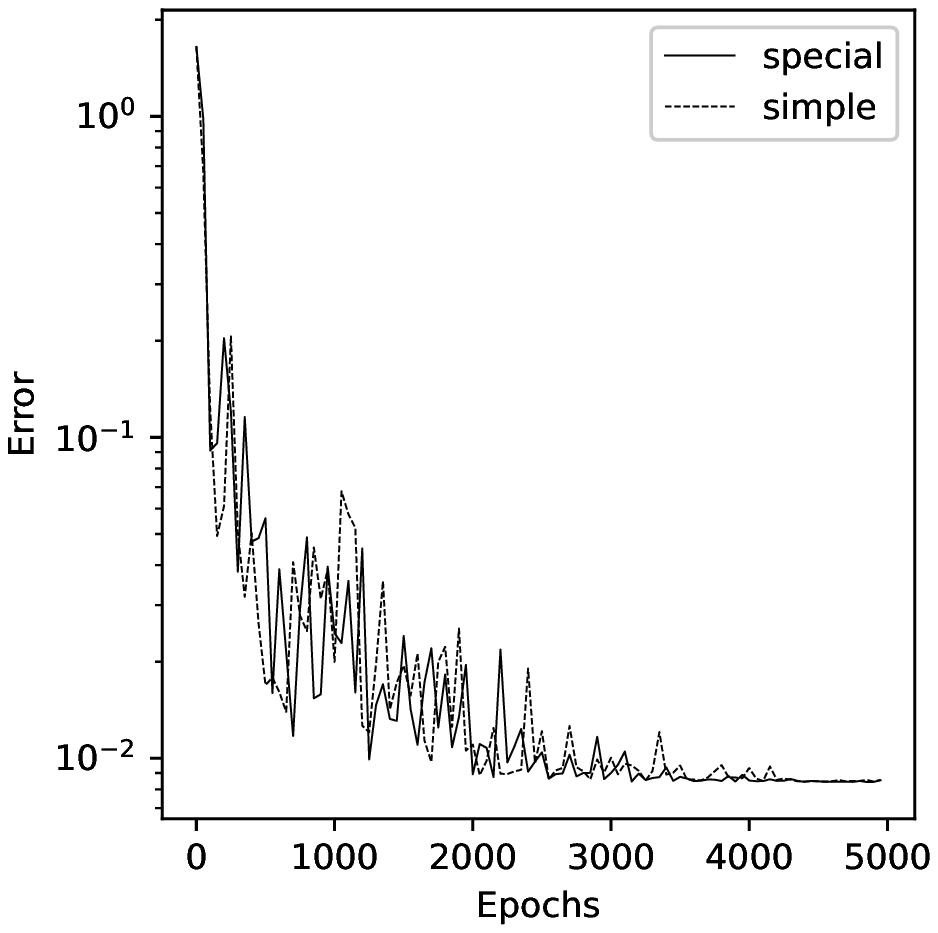}}
\subfloat[$s=0.9$]{
\includegraphics[scale=0.5]{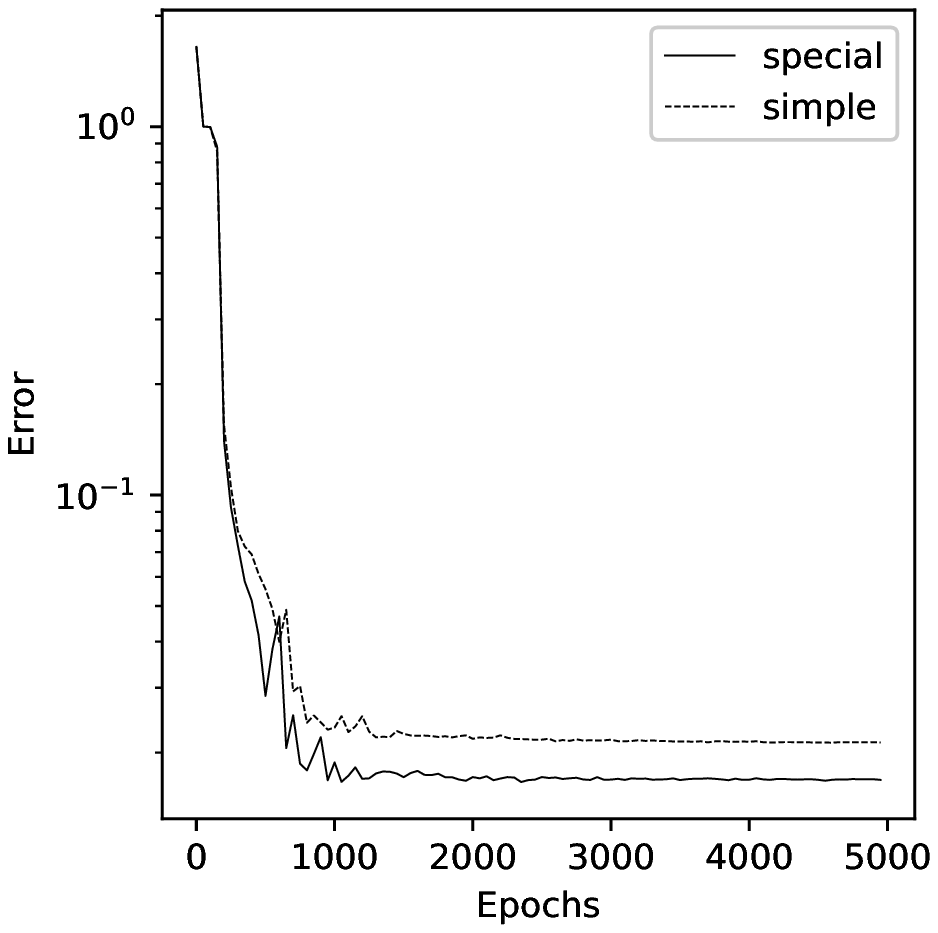}}
\caption{\em Errors $e_{\ell^2}(\mathcal{X})$ versus epochs for various $s$ using the special or simple structure when $d=3$.}
\label{Fig_case1_errors_d_3}
\end{figure}

\begin{table}
\centering
\begin{tabular}{|c|c|c|c|c|}
  \hline
   & \multicolumn{2}{c|}{Special structure in (\ref{05})} & \multicolumn{2}{c|}{Simple structure in (\ref{30})} \\\hline
  $s$ & Error & Running Time & Error & Running Time  \\\hline
  $0.1$ & $4.619\times10^{-2}$ & $8.7\times10^3$ & $5.402\times10^{-1}$ & $8.5\times10^3$\\\hline
  $0.3$ & $7.576\times10^{-3}$ & $9.9\times10^3$ & $1.800\times10^{-2}$ & $9.6\times10^3$\\\hline
  $0.5$ & $8.424\times10^{-3}$ & $1.3\times10^4$ & $1.257\times10^{-2}$ & $1.1\times10^4$\\\hline
  $0.7$ & $8.476\times10^{-3}$ & $1.8\times10^4$ & $8.498\times10^{-3}$ & $1.5\times10^4$\\\hline
  $0.9$ & $1.691\times10^{-2}$ & $2.5\times10^4$ & $2.134\times10^{-2}$ & $2.3\times10^4$\\\hline
\end{tabular}
\caption{\em Errors $e_{\ell^2}(\mathcal{X})$ and running time (seconds) for various $s$ using the special or simple structure when $d=3$.}
\label{Tab_case1_errors_d_3}
\end{table}

\begin{figure}
\centering
\includegraphics[scale=0.6]{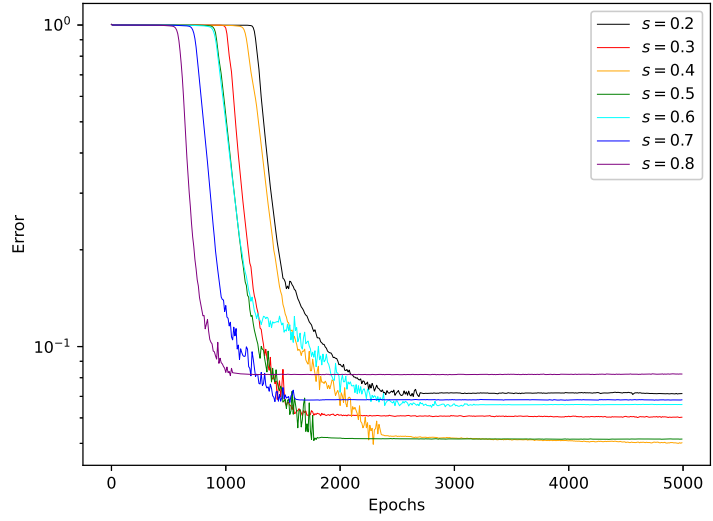}
\caption{\em Errors $e_{\ell^2}(\mathcal{X})$ versus epochs for various $s$ when $d=10$.}
\label{Fig_case1_errors_d_10}
\end{figure}

\begin{table}
\centering
\begin{tabular}{|c|c|c|}
  \hline
  $s$ & Error & Running Time\\\hline
  $0.2$ & $7.134\times10^{-2}$ & $2.4\times10^4$ \\\hline
  $0.3$ & $6.029\times10^{-2}$ & $2.5\times10^4$ \\\hline
  $0.4$ & $5.012\times10^{-2}$ & $2.7\times10^4$ \\\hline
  $0.5$ & $5.158\times10^{-2}$ & $2.9\times10^4$ \\\hline
  $0.6$ & $6.603\times10^{-2}$ & $3.4\times10^4$ \\\hline
  $0.7$ & $6.824\times10^{-2}$ & $4.1\times10^4$ \\\hline
  $0.8$ & $8.215\times10^{-2}$ & $5.6\times10^4$ \\\hline
\end{tabular}
\caption{\em Errors $e_{\ell^2}(\mathcal{X})$ and running time (seconds) for various $s$ when $d=10$.}
\label{Tab_case1_errors_d_10}
\end{table}

\section{Conclusion}
In summary, we develop a novel deep learning-based method to solve the spectral fractional Laplacian equation numerically. First, we reformulate the fractional equation as a regular partial differential equation of one more dimension by the Caffarelli-Silvestre extension. Next, we transform the extended problem to an equivalent minimal functional problem, and characterize the space of the weak solutions. To deal with the possible high dimensions, we employ FNNs to construct a special approximate class of the solution space, by which the approximate solution has consistent properties of the true solution. We then solve the minimization in the approximate class. In theory, we studied the approximation error of the special class under Barron-type hypothesis, and thereafter derive the solution error of this method. Finally, the effectiveness of the proposed deep Ritz method is illustrated in high-dimensional simulations.

In this work, we consider the error between the minimizer of the energy functional $\mathcal{I}$ in the approximation class and the true solution of the original problem. However, practically one needs to use numerical quadrature to compute $\mathcal{I}$, which brings extra errors. Consequently, future work may include the error analysis for the minimizer of the empirical loss function $\mathcal{I}_{\mathcal{T},\bar{h}}$ \eqref{10} discretized from $\mathcal{I}$ by the Monte Carlo method and sinc quadrature. More precisely, let $u^*$ be the true solution of \eqref{02} and
\begin{equation}
\hat{u}=\text{argmin}_{w\in\mathcal{N}}\mathcal{I}[w],\quad\hat{u}_{\mathcal{T},\bar{h}}=\text{argmin}_{w\in\mathcal{N}}\mathcal{I}_{\mathcal{T},\bar{h}}[w]
\end{equation}
for some DNN-based hypothesis space $\mathcal{N}$, then
\begin{equation}
\|u^*-\hat{u}_{\mathcal{T},\bar{h}}\|\leq\|u^*-\hat{u}\|+\|\hat{u}-\hat{u}_{\mathcal{T},\bar{h}}\|.
\end{equation}
Note that $\|u^*-\hat{u}\|$ has been investigated, it suffices to consider $\|\hat{u}-\hat{u}_{\mathcal{T},\bar{h}}\|$. Suppose $\mathcal{I}$ has a bounded inverse, then
\begin{multline}\label{31}
\|\hat{u}-\hat{u}_{\mathcal{T},\bar{h}}\|\leq\|\mathcal{I}^{-1}\|\left|\mathcal{I}[\hat{u}]-\mathcal{I}[\hat{u}_{\mathcal{T},\bar{h}}]\right|
\leq\|\mathcal{I}^{-1}\|\left(\mathcal{I}[\hat{u}_{\mathcal{T},\bar{h}}]-\mathcal{I}[\hat{u}]\right)\\
\leq\|\mathcal{I}^{-1}\|\left(\mathcal{I}[\hat{u}_{\mathcal{T},\bar{h}}]-\mathcal{I}[\hat{u}]+\mathcal{I}_{\mathcal{T},\bar{h}}[\hat{u}_{\mathcal{T},\bar{h}}]-\mathcal{I}_{\mathcal{T},\bar{h}}[\hat{u}_{\mathcal{T},\bar{h}}]\right)\\
\leq\|\mathcal{I}^{-1}\|\left[\left(\mathcal{I}[\hat{u}_{\mathcal{T},\bar{h}}]-\mathcal{I}_{\mathcal{T},\bar{h}}[\hat{u}_{\mathcal{T},\bar{h}}]\right)+\left(\mathcal{I}_{\mathcal{T},\bar{h}}[\hat{u}]-\mathcal{I}[\hat{u}]\right)\right],
\end{multline}
where we use the facts $\mathcal{I}[\hat{u}]\leq\mathcal{I}[\hat{u}_{\mathcal{T},\bar{h}}]$ and $\mathcal{I}_{\mathcal{T},\bar{h}}[\hat{u}_{\mathcal{T},\bar{h}}]\leq\mathcal{I}_{\mathcal{T},\bar{h}}[\hat{u}]$. On the right hand side of \eqref{31}, the first term  describes the generalization error of the empirical loss minimization over the hypothesis space, and the second term characterizes the bias coming from the numerical quadrature of the integrals. In \cite{Lu2021}, related analysis is conducted for Poisson equation and Schr\"{o}dinger equation using Rademacher complexity, but it is only for the Monte-Carlo quadrature. It is promising and challenging to do similar analysis for our method on fractional Laplacian equations, which involves both the stochastic Monte-Carlo method and the deterministic sinc quadrature for approximation.

\bibliography{expbib}
\bibliographystyle{plain}
\end{document}